\documentclass[11pt, oneside]{amsart}

\usepackage[british,english]{babel} 
\usepackage{graphics,color,pgf}
\usepackage{epsfig}
\usepackage[ansinew]{inputenc}
\usepackage[all]{xy}
\usepackage{hyperref}
\newdir{ >}{!/8pt/@{}*@{>}}
\usepackage{amssymb, amsmath,amsthm, mathtools, amscd}
\usepackage{mathrsfs}
\usepackage{stmaryrd}
\usepackage[margin=1.5in]{geometry}

\theoremstyle{plain}
\newtheorem{teor}{Theorem}[section]
\newtheorem{lem}[teor]{Lemma}

\newtheorem{prop}[teor]{Proposition}

\theoremstyle{definition}
\newtheorem{deft}[teor]{Definition}

\theoremstyle{remark}
\newtheorem{oss}[teor]{Remark}

\DeclareMathOperator\tr{tr}

\DeclareMathOperator\grad{grad}

\DeclareMathOperator\bbH{\mathbb{H}}

\DeclareMathOperator\bbN{\mathbb{N}}

\DeclareMathOperator\bbR{\mathbb{R}}

\DeclareMathOperator\calB{\mathcal{B}}

\DeclareMathOperator\calM{\mathcal{M}}
\DeclareMathOperator\calN{\mathcal{N}}

\DeclareMathOperator\calX{\mathcal{X}}

\DeclareMathOperator\po{\textup{PO}}

\DeclareMathOperator\jac{\textup{Jac}}
\DeclareMathOperator\vol{\textup{Vol}}
\DeclareMathOperator\barb{\textup{bar}_{\mathcal{B}}}

\title[Equivariant maps for measurable cocycles]{Equivariant maps for measurable cocycles with values into higher rank Lie groups}

\author[]{A. Savini}
\address{Section de Math\'ematiques, University of Geneva, Rue du Conseil-G\'en\'eral 7-9, 1205 Geneva, Switzerland}
\email{Alessio.Savini@unige.ch}
\thanks{}

\date{\today.\ \copyright{\ The author was partially supported by the FNS grant no. 200020-192216.}}

\begin{document}

\begin{abstract}
Let $G$ be a semisimple Lie group of non-compact type and let $\calX_G$ be the Riemannian symmetric space associated to it.
Suppose $\calX_G$ has dimension $n$ and it does not contain any factor isometric to either $\bbH^2$ or $\textup{SL}(3,\bbR)/\textup{SO}(3)$. Given a closed $n$-dimensional complete Riemannian manifold $N$, let $\Gamma=\pi_1(N)$ be its fundamental group and $Y$ its universal cover. Consider a representation $\rho:\Gamma \rightarrow G$ with a measurable $\rho$-equivariant map $\psi:Y \rightarrow \calX_G$. Connell-Farb described a way to construct a map $F:Y \rightarrow \calX_G$ which is smooth, $\rho$-equivariant and with uniformly bounded Jacobian.

In this paper we extend the construction of Connell-Farb to the context of measurable cocycles. More precisely, if $(\Omega,\mu_\Omega)$ is a standard Borel probability $\Gamma$-space, let $\sigma:\Gamma \times \Omega \rightarrow G$ be measurable cocycle. We construct a measurable map $F:Y \times \Omega \rightarrow \calX_G$ which is $\sigma$-equivariant, whose slices are smooth and they have uniformly bounded Jacobian. For such equivariant maps we define also the notion of volume and we prove a sort of mapping degree theorem in this particular context. 
\end{abstract}

\maketitle

\section{Introduction}

The barycenter construction appeared for the first time in the paper of Douady and Earle \cite{douady:earle} who wanted to extend self-maps of the circle to the whole Poincar\'e disk. So far this technique has been widely developed and it has been fruitfully used to obtain several strong rigidity statements in geometric topology. For instance Besson-Courtois-Gallot \cite{bcg95,bcg96,bcg98} used the barycenter method to prove the minimal entropy conjecture in the case of rank-one locally simmetric manifolds. More precisely, given a continuous map $f:N \rightarrow M$ between compact rank-one manifolds, the authors construct the so-called \emph{natural maps} by applying the barycenter to a family of measures that are equivariant with respect to the induced morphism $\pi_1(f)$. Natural maps are smooth, equivariant maps whose Jacobian is uniformly bounded by $1$ and the equality at a point is attained if and only if the differential on the tangent space is a homothety. Similar applications of natural maps in the study of real hyperbolic manifolds were given for instance by Boland-Connell-Souto \cite{boland05} and by Francaviglia and Klaff \cite{franc06:articolo,franc09}. The latter constructed the natural map associated to a representation $\rho:\Gamma \rightarrow \po(m,1)$, where $\Gamma \leq \po(n,1)$ is a torsion-free lattice and $m \geq n \geq 3$. The existence of such a maps allowed to define the notion of volume of representations and to show that this numerical invariant is rigid. Indeed, it holds $\vol(\rho) \leq \vol(\Gamma \backslash \bbH^n_{\bbR})$ for every representation $\rho:\Gamma \rightarrow \po(m,1)$ and the equality is attained if and only if the representation $\rho$ is discrete and faithful. Successively, the author extended the same notion to the context of complex and quaternionic lattices getting a stronger rigidity phenomenon. Indeed, as shown by the author and Francaviglia, the volume function is actually rigid also at the ideal points of the character variety \cite{savini:articolo,savini2:articolo}, leading to a proof of Guilloux's conjecture \cite[Conjecture1]{guilloux:articolo} for $n=2$.

The attempt to extend the proof of the minimal entropy conjecture to semisimple Lie groups of higher rank led Connell-Farb \cite{connellfarb2,connellfarb1,connellfarb4} to define natural maps also in this different context. This strategy allowed the authors to prove the conjecture for manifolds which are quotients of products of rank-one symmetric spaces. Similary they extended the mapping degree theorem for continuous maps between higher rank manifolds. Under the higher rank assumption, it is worth mentioning also the volume rigidity for representations of lattices obtained by Kim and Kim \cite{kim:kim:14} via continuous bounded cohomology.

Among the other possible applications of the barycenter construction and natural maps, it is worth mentioning the rigidity result obtained by Boland-Connell \cite{boland:connell} for foliations of Riemannian manifolds with negatively curved leaves and the rigidity phenomena proved by Boland-Newberger \cite{boland:newberger} and by Adeboye-Bray-Constantine \cite{bray} and the author \cite{savini_conv_proj} for Finsler/Benoist manifolds. To conclude this historical introduction, we recall also the work of Lafont-Schmidt \cite{Lafont-Schmidt}. Using the barycenter construction they showed the positivity of simplicial volume of locally symmetric manifolds of higher rank and the surjectivity of the comparison map in bounded cohomology for a specific range of indices. 

As already done by the author for measurable cocycles of rank-one lattices \cite{savini4:articolo,savini:tautness}, in this paper we would like to apply the barycenter to build natural maps for measurable cocycles taking values into higher rank Lie groups. Let $G$ be a semisimple Lie group of non-compact type with rank bigger than or equal to $2$ and let $\calX_G$ be the Riemannian symmetric space associated to it. Suppose $\calX_G$ has no factor isometric to either $\bbH^2$ or $\textup{SL}(3,\bbR)/\textup{SO}(3)$. If we denote by
$n=\textup{dim}(\calX_G)$ the dimension, we are going to show the following 

\begin{teor}\label{teor:natural:map}
Let $N$ be a closed $n$-dimensional complete Riemannian manifold with fundamental group $\Gamma=\pi_1(N)$ and universal cover $Y$. Let $(\Omega,\mu_\Omega)$ be a standard Borel probability $\Gamma$-space. Let $\sigma:\Gamma \times \Omega \rightarrow G$ be a measurable cocycle. Then there exists a measurable map $F:Y \times \Omega \rightarrow \calX_G$ which is $\sigma$-equivariant, whose slice $F_x:Y \rightarrow \calX_G$ is differentiable for almost every $x \in \Omega$ and there exists a constant $C>0$ such that
$$
\jac_a F_x < C \ ,
$$
for every $a \in Y$ and almost every $x \in \Omega$. Here $C$ is a constant depending only on the dimension $n$ and on the geometry of both $Y$ and $\calX_G$. 
\end{teor}
 
%The first comment regards the assumption about the existence of a measurable $\sigma$-equivariant map, which is not so strong. Indeed, by the hypothesis we have on both $\Gamma$ and $\Omega$, there exists a measurable $\Gamma$-fundamental domain $\Delta_\Gamma$ in $\Omega$. Then we can define a measurable function $\psi:\Delta_\Gamma \rightarrow \textup{Meas}(\calX_G,\calX_G)$ and then extend it to the whole $\Omega$ using the equivariance property. Notice additionally that in the case of rank-one Lie groups, the previous setting allows us to extend the construction of natural maps to cocycles that do not necessarily admit a boundary map, but losing the sharpness on the estimate of the Jacobian. 

Theorem \ref{teor:natural:map} should be interpreted as a generalization of Connell-Farb theorem to the wider context of measurable cocycle theory. The proof will be based crucially on the existence of a measurable equivariant map proved in Lemma \ref{lemma:measurable:map}. Indeed we are going to consider the pushforward of a suitable equivariant family of measures on $Y$ and then we are going to take the convolution with the Patterson-Sullivan density associated to $G$ (see \cite{patterson76,sullivan79,albuquerque97,albuquerque99}). Since this convolution is fully supported on the \emph{Furstenberg-Poisson boundary} $B(G)$ of $G$, we can apply correctly the barycenter to get our desired map. The computation on slices are exactly the one made by Connell-Farb \cite{connellfarb2,connellfarb1,connellfarb4}. Notice that when $\Gamma$ is a higher rank lattice, the existence of such a map can be argued by Zimmer Superrigidity Theorem \cite{zimmer:annals}, since the cocycle may be trivialized.  

Given a measurable cocycle $\sigma:\Gamma \times \Omega \rightarrow G$, let now assume to have a measurable map $\Phi:Y \times \Omega \rightarrow \calX_G$ which is $\sigma$-equivariant and whose slices $\Phi_x:Y \rightarrow \calX_G, \hspace{5pt} \Phi_x(a):=\Phi(a,x)$ are smooth for almost every $x \in \Omega$. If the Jacobian of the slices is uniformly bounded (that is $\Phi$ has \emph{essentially bounded slices}), then we define the notion of \emph{volume $\vol(\Phi)$ of the measurable map $\Phi$}. Notice that an example of such a map is exactly the natural map we constructed. Additionally, in the particular case of Zariski dense cocycles of higher rank lattices, the volume boils down to the covolume of the lattice itself. 

Since, given a continuous function between compact manifolds allows to pullback measurable cocycles, we state a result which should be interpreted as a mapping degree theorem for measurable equivariant maps. More precisely we have

\begin{prop}\label{prop:degree:map}
Let $N,M$ be a closed $n$-dimensional Riemannian manifolds with fundamental groups $\Gamma=\pi_1(N), \Lambda=\pi_1(M)$ and universal covers $Y,X$, respectively. Suppose that there exists a smooth function $f:N \rightarrow M$ with non-vanishing degree and uniformly bounded Jacobian. Let $(\Omega,\mu_\Omega)$ be a standard Borel probability $\Lambda$-space and let $\sigma:\Lambda \times \Omega \rightarrow G$ be a measurable cocycle. Given a measurable $\sigma$-equivariant map $\Phi:Y \times \Omega \rightarrow \calX_G$ with smooth essentially bounded slices, it holds that
$$
|\deg(f)| \leq \frac{\vol(f^\ast \Phi)}{\vol(\Phi)} \ .
$$
\end{prop}

\subsection*{Plan of the paper} 

In Section \ref{sec:preliminary:def} we recall basic definitions and results that we need for our exposition. We start with Section \ref{sec:zimmer:cocycle} where we remind the notion of measurable cocycle, cohomology class and equivariant map. Then we move to Section \ref{sec:patterson:sullivan} where we describe the Patterson-Sullivan density associated to a higher rank Lie group. Section \ref{sec:bar:natural} is devoted to the description of the barycenter method and to the definition of Connell-Farb natural map. In Section \ref{sec:natural:map} we prove Theorem \ref{teor:natural:map} and we compare our definition of natural map with the one of Connell-Farb (Proposition \ref{prop:natural:rep}). We then show in Proposition \ref{prop:natural:cohomology} how natural maps vary in a specific cohomology class. The definition of volume of a equivariant map is given in Section \ref{sec:volume}, where we prove also Proposition \ref{prop:degree:map}. 

\subsection*{Acknowlegdements} I am grateful to Marco Moraschini for the useful discussion about this topic. 

I would also like to thank the anonymous referee for her/his suggestions that allowed me to improve the quality of the paper. 

\section{Preliminary definitions and results}\label{sec:preliminary:def}
 
In this section we are going to recall all the definitions and the results we are going to need throughout the paper. We will first give a brief introduction about the notion of measurable cocycle. Then we will introduce a key tool in order to construct our natural maps: the Patterson-Sullivan family of measures associated to a higher rank semisimple Lie group. This family will generalize the standard construction made by both Patterson \cite{patterson76} and Sullivan \cite{sullivan79} in case of rank-one Lie groups of non-compact type. Finally we are going to recall the barycenter construction of a probability measure supported on the Furstenberg-Poisson boundary.

\subsection{Measurable cocycles}\label{sec:zimmer:cocycle}

In this section we are going to recall the main definition of measurable cocycles. The following will be a short introduction and we refer to both Furstenberg \cite{furst:articolo73,furst:articolo} and to Zimmer \cite{zimmer:preprint,zimmer:libro} for a more detailed description. 

Let $G,H$ be two locally compact second countable groups and endow both with their Haar $\sigma$-algebras and measurable structures. Fix a standard Borel probability space $(\Omega,\mu)$ where $\mu$ has no atoms. If $G$ acts on $\Omega$ by measure preserving transformations, we are going to call $(\Omega,\mu)$ a \emph{standard Borel probability $G$-space}. If $(\Theta,\nu)$ is another measure space, we denote by $\textup{Meas}(\Omega,\Theta)$ the space of measurable maps endowed with the topology of convergence in measure. 

With the notation above, we are now ready to give the following 

\begin{deft}
A measurable map $\sigma:G \times \Omega \rightarrow H$ is a \emph{measurable cocycle} (or \emph{Zimmer's cocycle}) if it holds
\begin{equation}\label{eq:zimmer:cocycle}
\sigma(g_1 g_2,x)=\sigma(g_1,g_2.x) \sigma(g_2,x) \ ,
\end{equation}
for every $g_1,g_2 \in G$ and almost every $x \in \Omega$. Here the notation $g_2.x$ refers to the action of $G$ on the space $\Omega$. 
\end{deft}

It is worth noticing that Equation (\ref{eq:zimmer:cocycle}) can be suitably interpreted as a sort of generalization of the chain rule for derivatives in this context. For the reader who is familiar with group cohomology, we want to underline that by viewing $\sigma \in \textup{Meas}(G,\textup{Meas}(\Omega,H))$, Equation (\ref{eq:zimmer:cocycle}) is equivalent to require that $\sigma$ is a Borel $1$-cocycle in the sense of Eilenber-MacLane (see \cite{feldman:moore,zimmer:preprint}). Using the latter interpretation, we can naturally ask which is the right condition on two cocycles for being cohomologous.

\begin{deft}\label{def:cocycle:cohomology}
Let $\sigma:G \times \Omega \rightarrow H$ be a measurable cocycle. Given a measurable map $f:\Omega \rightarrow H$, the \emph{twisted cocycle with respect to $f$ and $\sigma$} is given by
$$
f.\sigma:G \times \Omega \rightarrow H, \hspace{5pt} (f.\sigma)(g,x):=f(g.x)^{-1}\sigma(g,x)f(x) \ , 
$$
for every $g \in G$ and almost every $x \in \Omega$. Two measurable cocycles $\sigma_1,\sigma_2:G \times \Omega \rightarrow H$ are \emph{cohomologous} (or \emph{equivalent}) if there exists a measurable map $f:\Omega \rightarrow H$ such that 
$$
\sigma_2=f.\sigma_1 \ . 
$$
\end{deft}

The role played by measurable cocycles in mathematics is central. We will particularly be interested in the examples coming from representation theory. 

\begin{deft}\label{def:cocycle:rep}
Let $\rho:G \rightarrow H$ be a continuous representation. Fix any standard Borel probability $G$-space $(\Omega,\mu)$. The \emph{measurable cocycle associated to $\rho$} is defined as
$$
\sigma_\rho:G \times \Omega \rightarrow H, \hspace{5pt} \sigma_\rho(g,x):=\rho(g) \ ,
$$
for every $g \in G$ and almost every $x \in \Omega$. 
\end{deft}

The above definition should suggest how representation theory can be suitably seen inside the wider context of measurable cocycles theory. We want to underline that even if the definition above depends actually also on the choice of Borel space $\Omega$, we prefer to omit this dependence from the notation $\sigma_\rho$. Additionally, it is worth noticing that when $G$ is a discrete group every representation is automatically continuous.  

We conclude this short introduction about measurable cocycles by recalling the notion of equivariant maps and how they change along cohomology classes. 

\begin{deft}\label{def:equivariant:map}
Let $\sigma:G \times \Omega \rightarrow H$ be a measurable cocycle. Assume that $G$ and $H$ acts continuously on two topological spaces $Y$ and $X$, respectively. A measurable map $\psi:Y \times \Omega \rightarrow X$ is \emph{$\sigma$-equivariant} if it holds
$$
\psi(g.a,g.x)=\sigma(g,x)\psi(a,x) \ ,
$$ 
for every $g \in G$ and almost every $a \in Y, x \in \Omega$. 
\end{deft}

Assume that $\sigma:G \times \Omega \rightarrow H$ is a measurable cocycle and let $\psi:Y \times \Omega \rightarrow X$ be a measurable $\sigma$-equivariant map as above. Given a measurable map $f:\Omega \rightarrow H$, we can define the following map 
$$
f.\psi:Y \times \Omega \rightarrow X, \hspace{5pt} (f.\psi)(a,x)=f(x)^{-1}\psi(a,x) \ ,
$$
for almost every $a \in Y,x \in \Omega$. It is easy to verify that the map $f.\psi$ is a measurable $f.\sigma$-equivariant map. 

\subsection{Patterson-Sullivan measures}\label{sec:patterson:sullivan}

In this section we are going to recall the definition of Patterson-Sullivan density associated to a lattice is a semisimple Lie group of non-compact type. For rank-one Lie groups we mainly refer to the pioneering work of Patterson \cite{patterson76} and Sullivan \cite{sullivan79}. In the case of negatively curved spaces we suggest the paper of Burger-Mozes \cite{burger:mozes}. However, since we will be mainly interested in the case of higher rank Lie groups, we refer the reader to Albuquerque \cite{albuquerque97,albuquerque99} for a more detailed description of the argument.

Let $G$ be a semisimple Lie group of non-compact type and let $\calX_G$ the Riemannian symmetric space associated to $G$. We denote by $\partial_\infty \calX_G$ the boundary at infinity of $\calX_G$ endowed with the cone topology. Given a point $a \in \calX_G$, the \emph{Busemann function pointed at $a$} is the map
$$
\beta_a:\calX_G \times \partial_\infty \calX_G \rightarrow \bbR, \hspace{5pt} \beta_a(b,\xi):=\lim_{t \to \infty} d_G(c(t),a)-d_G(c(t),b) \ ,
$$
where $d_G(\cdot,\cdot)$ is the distance associated to the Riemannian structure on $\calX_G$ and $c:[0,\infty) \rightarrow \calX_G$ is the unique geodesic ray starting at $c(0)=a$ and ending at $c(\infty)=\xi$. Busemann functions are convex and this convexity property will be crucial to apply correctly the barycenter construction, as we will see later in Section \ref{sec:bar:natural}.

\begin{deft}
Fix any basepoint $a \in \calX_G$. The \emph{critical exponent $\delta_G$} associated to a semisimple Lie group $G$ of non-compact type is given by
$$
\delta_G:=\inf \{ s \in \bbR | \int_{G} e^{-sd_G(a,g.a)}d\mu_G(g) < \infty \} \ ,
$$
where $\mu_G$ is the Haar measure on $G$. It is worth noticing that the definition we gave does not depend on the particular choice of the basepoint $a \in \calX_G$. 
\end{deft}

Recall that the \emph{volume entropy} of the symmetric space $\calX_G$ is given by 
$$
h(\calX_G):=\lim_{r \to \infty} \frac{\log(\vol(B_r(a)))}{r} \ ,
$$
where $B_r(a)$ is the Riemannian ball pointed at $a$ of radius $r$ and $\vol$ is the standard Riemannian volume on $\calX_G$. By Albuquerque \cite[Theorem 2]{albuquerque97}, \cite[Theorem C]{albuquerque99} we know that the critical exponent of $G$ it is equal to the critical exponent of any of its lattices and it coincides with the volume entropy, that is
$$
\delta_G = h(\calX_G) \ . 
$$

We are now ready to give the definition of Patterson-Sullivan family associated to the group $G$. This will actually be included in the more general definition of conformal density. In order to proceed, given any topological space $X$, we are going to denote by $\calM^1(X)$ the space of positive probability measure on $X$. 

\begin{deft}\label{def:conformal:density}
Let $\alpha>0$ be a positive real number. An $\alpha$-\emph{conformal density} is a measurable map 
$$
\nu:\calX_G \rightarrow \calM^1(\partial_\infty \calX_G), \hspace{5pt} \nu(a):=\nu_a \ ,
$$
such that
\begin{enumerate}
	\item each measure $\nu_a$ has no atoms;
\item given two points $a,b \in \calX_G$, the measures $\nu_a$ and $\nu_b$ are absolutely continuous and it holds
$$
\frac{d\nu_a}{d\nu_b}(\xi)=e^{-\alpha \beta_b(a,\xi)} \ ,
$$
where $\xi \in \partial_\infty \calX_G$ and $\beta_b(a,\xi)$ is the Busemann function pointed at $b \in \calX_G$. 
\end{enumerate}

A \emph{Patterson-Sullivan density} is the $h(\calX_G)$-conformal density.   
\end{deft}

The existence of a Patterson-Sullivan density is proved by Albuquerque, who proves also that such a density is essentially unique up to a multplicative constant \cite[Proposition D]{albuquerque99}.  

The last remarkable property of the Patterson-Sullivan measures refers to their support. More precisely the support of the measure $\nu_a$ coincides with the \emph{Furstenberg-Poisson boundary} $B(G)$ of $G$, as shown in \cite[Theorem]{albuquerque97}, \cite[Theorem C]{albuquerque99}. The latter can be seen as the unique $G$-orbit of a regular point in $\partial_\infty \calX_G$ and it is usually identified with the homogeneous space $G/P$, where $P$ is any minimal parabolic subgroup. Notice that when the rank of $G$ is equal to $1$ every point in $\partial_\infty \calX_G$ is regular and the Furstenberg-Poisson boundary is equal to the boundary at infinity. 

We conclude by underling that the measure $\nu_a$ of the Patterson-Sullivan density associated to $G$ is the unique probability measure on $B(G)$ which is $K_a:=\textup{Stab}_G(a)$-invariant. The previous remark guarantees also the fact that the density $\nu$ is a $G$-equivariant map, that is $\nu_{g.a}=g_\ast \nu_a$, where $g_\ast \nu_a$ is the push-forward measure. 

\subsection{Barycenter and Connell-Farb construction}\label{sec:bar:natural}
The main subject of this section will be the barycenter construction introduced by Douady-Earle \cite{douady:earle}. This construction was exploited by Besson-Courtois-Gallot \cite{bcg95,bcg96,bcg98} to construct natural maps for rank-one Lie groups of non-compact type. The same approach was extended by Connell-Farb \cite{connellfarb2,connellfarb1,connellfarb4} to higher rank Lie groups. 

Let $G$ be a semisimple Lie group of non-compact type and let $\calX_G$ the Riemannian symmetric space associated to $G$. As before, denote by $\partial_\infty \calX_G$ the boundary at infinity of $\calX_G$. Let $\nu$ be a positive probability measure on $\partial_\infty \calX_G$, that is $\nu \in \calM^1(\partial_\infty \calX_G)$. If we fix a basepoint $o \in \calX_G$, using the measure $\nu$ we can define the map 

$$
\calB_\nu: \calX_G \rightarrow \bbR, \hspace{5pt} \calB_\nu(a):=\int_{\partial_\infty \calX_G} \beta_o(a,\xi)d\nu(\xi) \ ,
$$ 
where $\beta_o$ is the Busemann function pointed at $o \in \calX_G$ (see Section \ref{sec:patterson:sullivan}). Even if a priori the Busemann function is not strictly convex, since we are not necessarily considering the rank one case, under suitable hypothesis we can say something abouth the convexity of the function $\calB_\nu$. As shown by Connell-Farb \cite[Proposition 12]{connellfarb3}, when the measure $\nu$ is fully supported on the Furstenberg-Poisson boundary $B(G)$, then the function $\calB_\nu$ is strictly convex. Hence there exists a unique point which attains the minimun. 

\begin{deft}\label{def:barycenter:measure}
Let $\nu \in \calM^1(\partial_\infty \calX_G)$ be a positive probability measure whose support coincides with the Furstenberg-Poisson boundary of $\calX_G$, that is $\textup{supp}(\nu)=B(G)$. Then the \emph{barycenter} of the measure $\nu$ is the point in the symmetric space $\calX_G$ defined as 
$$
\barb(\nu):=\textup{argmin}(\calB_\nu) \ ,
$$
where argmin is the point where $\calB_\nu$ attains its minimum.
\end{deft}

The subscript $\calB$ we used in the definition of the barycenter suggests the dependence of the construction on the function $\calB_\nu$, and hence on Busemann functions.

We report below a brief list of properties of the barycenter.
\begin{itemize}
	\item The barycenter is weak-${}^\ast$ continuous. More precisey given a sequence $(\nu_k)_{k \in \bbN}$ of probability measures such that $\nu_k$ converges to $\nu$ in the weak-${}^\ast$ topology and they are all supported on the boundary $B(G)$, it holds
$$
\lim_{k \to \infty} \barb(\nu_k)=\barb(\nu) \ . 
$$
 	\item The barycenter is $G$-equivariant. Given an element $g \in G$ and a probability measure $\nu$ supported on $B(G)$, we have
$$
\barb(g_\ast \nu)=g\barb(\nu) \ ,
$$
where $g_\ast \nu$ denotes the pushforward measure with respect to $g$. 
	\item The barycenter of a probability measure $\nu$ supported on $B(G)$ satisfies the following implicit equation 
\begin{equation}\label{eq:implicit:barycenter}
\int_{B(G)} d\beta_o|_{(\barb(\nu),\xi)}(\cdot)d\nu(\xi)=0 \ ,
\end{equation}
where $d\beta_o$ denotes the differential of the Busemann function $\beta_o$. 
\end{itemize}

Since we will need to compare it with our version of natural map in the case of measurable cocycle, we will conclude the section by recalling briefly the Connell-Farb approach to natural maps \cite{connellfarb2,connellfarb1,connellfarb4}. Let $N$ be a closed $n$-dimensional complete Riemannian manifold whose fundamental group is $\Gamma=\pi_1(N)$ and with universal cover $Y$. 
Fix first a positive real number $s>h(Y)$, where $h(Y)$ is the volume entropy of $Y$. Denoting by $\mu$ the Riemmian volume measure on $Y$ we can define the following family of measures in terms of the Radon-Nikodym derivative
\begin{equation}\label{eq:equivariant:family}
\frac{d\mu^s_a}{d\mu}:= \frac{e^{-sd_Y(a,z)}}{\int_{Y} e^{-sd_Y(a,z)}d\mu(z)} \ ,
\end{equation}
where $a \in Y$ and $d_Y(\cdot,\cdot)$ stands for the Riemannian distance on $Y$. This is the same family defined for instance in \cite{connellfarb2} and it is clearly equivariant with respect to the natural action of $\Gamma$, that is
$$
\mu^s_{\gamma a}=\gamma_\ast(\mu^s_a) \ ,
$$
where $\gamma \in \Gamma$ and $\gamma_\ast$ is the pushforward measure. 

Let now $\rho:\Gamma \rightarrow G$ be a representation. Consider a measurable map $\psi:Y \rightarrow \calX_G$ which is $\rho$-equivariant. Then one can define the following family of measures
\begin{equation}\label{eq:convolution:measure}
\lambda^s_a:=((\psi)_\ast(\mu^s_a)) \ast \{\nu_b \}_{b \in \calX_G} \ ,
\end{equation}
for every $a \in Y$. The convolution that appears in Equation (\ref{eq:convolution:measure}) is defined as follows
$$
\lambda^s_a(U):=\int_{\calX_G} \nu_b(U)d((\psi)_\ast(\mu^s_a))(b)=\int_{Y}\nu_{\psi(z)}(U)d\mu^s_a(z) \ , 
$$
for every Borel subset $U \subseteq B(G)$. Since we used the Patterson-Sullivan family in the convolution, for every $a \in Y$ the measure $\lambda^s_a$ is supported on the boundary $B(G)$. Thus we can correctly apply the barycenter to get a point $\calX_G$. Indeed Connell-Farb \cite{connellfarb2,connellfarb1} define the map 
$$
F^s:Y \rightarrow \calX_G \  ,
$$
$$
F^s(a):=\barb(\lambda^s_a) = \barb \left( \left( \int_{Y} e^{-sd_Y(a,z)-h(\calX_G)\beta_o(\psi(z),\xi)}d\mu(z) \right) d\nu_o(\xi) \right) \ .
$$
If we now substitute this expression in the implicit Equation (\ref{eq:implicit:barycenter}) we obtain
\begin{equation}
\int_{B(G)} d\beta_o|_{(F^s(a),\xi)} ( \cdot )d\lambda^s_a(\xi)=0 \ ,
\end{equation}
and by differentiating it we get
\begin{align}\label{eq:implicit:differentiated:representation}
&\int_{B(G)} \nabla d\beta_o|_{(F^s(a),\xi)}(D_aF^s(u),v)d\lambda^s_a(\xi)=\\
s\int_{Y} &\int_{B(G)} d\beta_o|_{(F^s(a),\xi)}(v) \cdot \langle \grad_a d_Y(a,z), u \rangle d\nu_{\psi(z)}(\xi)d\mu^s_a(z) \nonumber \ ,
\end{align}
for every $a \in Y$, $u \in T_aY$ and $v \in T_{F^s(a)}\calX_G$. Here $\nabla$ is the Levi-Civita connection associated to the standard Riemannian structure on $\calX_G$ and $\grad_a$ is the Riemannian gradient on $Y$. In the next section we will see that the equation above will hold at every slice of our natural map associated to a fixed measurable cocycle. 

\section{Natural maps associated to measurable cocycles of higher rank lattices}\label{sec:natural:map}

Let $G$ be a semisimple Lie group of non-compact type with rank bigger than or equal to $2$ and denote by $\calX_G$ the associated symmetric space. If $\textup{dim}(\calX_G)=n$, let $N$ be a closed $n$-dimensional complete Riemannian manifold with fundamental group $\Gamma=\pi_1(N)$ and universal cover $Y$. In this section we are going to construct explicitly natural maps associated to Zimmer's cocycles valued into $G$. The main strategy to construct these maps will be to consider a suitable equivariant family of measures on $Y$, consider their pushforward with respect a measurable $\sigma$-equivariant map and then apply the convolution with the Patterson-Sullivan family introduced in Section \ref{sec:patterson:sullivan}. The existence of natural maps, that is $\sigma$-equivariant maps with differentiable slices and uniformly bounded Jacobian, will generalize the construction already developed by the author \cite{savini4:articolo} for torsion-free lattices in rank-one Lie groups. However it is worth noticing that the approach we are going to develop here is quite different with respect to the one of \cite{savini4:articolo}. Indeed here we are going to use measurable equivariant map defined on $Y$ and on the symmetric space $\calX_G$ rather then boundary maps.

Before proving the existence of a measurable equivariant map, we need first to recall the definition of measurable fundamental domain with respect to the action of $\Gamma$ on $Y$. A measurable subset $\Delta_\Gamma \subset Y$ is a \emph{measurable fundamental domain} if it holds 
$$\mu(\Delta_\Gamma \cap \gamma \Delta_\Gamma)=0 \ ,$$ 
for every non trivial element $\gamma \in \Gamma$, and
$$\mu(Y \setminus \bigcup_{\gamma \in \Gamma} \gamma \Delta_\Gamma)=0 \ .$$
Recall that $\mu$ is the measure induced by the Riemannian structure on $Y$. In literature one can impose more restrictive conditions to define a measurable fundamental domain (for instance one may require that the above equations hold everywhere and not only almost everywhere). Nevertheless for our purposes it is sufficient to deal with the definition we gave, since we will care only about functions defined almost everywhere (for instance the equivariance must hold only on a full measure subset and, similarly, the construction of the natural map is not affected if we change along a measure zero subset the starting equivariant function).  

\begin{lem}\label{lemma:measurable:map}
Let $\Gamma=\pi_1(N)$ be the fundamental group of a closed $n$-dimensional complete Riemannian manifold $N$ and let $Y$ be its universal cover. Fix a standard Borel probability $\Gamma$-space. Given a measurable cocycle $\sigma:\Gamma \times \Omega \rightarrow G$, there exists a measurable $\sigma$-equivariant map 
$$
\psi:Y \times \Omega \rightarrow \calX_G. 
$$
\end{lem}

\begin{proof}
Recall that $\Gamma$ acts freely, properly discontinuously and by isometries on $Y$ (it is worth noticing that $\Gamma$ is actually lattice in the isometry group $\textup{Isom}(Y)$, where the latter is endowed with the compact-open topology, see for instance \cite[Lemma 4.2]{Loh-Sauer} for a detailed proof). For such an action a measurable fundamental domain exists (an explicit example of measurable fundamental domain is the one given by the Dirichlet condition, that is 
$$
\Delta_\Gamma:=\{a \in Y \ | \ d_Y(a,o)<d_Y(a,\gamma.o), \ \textup{for every $\gamma \in \Gamma \setminus \{e_\Gamma\}$} \} \ ,
$$
where $o$ is a fixed based point in $Y$).

Given a measurable fundamental domain $\Delta_\Gamma$, consider a measurable function $q:\Delta_\Gamma \times \Omega \rightarrow \calX_G$. We can get a measurable map $\psi:Y \times \Omega \rightarrow \calX_G$ as follows
$$
\psi(a,x):=
\begin{cases*}
q(a,x),& \ \ \text{if $a \in \Delta_\Gamma$}\ ,\\
\sigma(\gamma,x_0)q(a_0,x_0), & \ \ \text{if $(a,x)=\gamma.(a_0,x_0)$}.\\
\end{cases*}
$$
The function $\psi$ is well-defined since $\Delta_\Gamma$ is a measurable fundamental domain and $\Gamma$ acts on $\Omega$ by measure preserving transformations. It is worth noticing that $\psi$ could actually be an almost everywhere defined function (for instance if we consider the Dirichlet condition). In that case one may extend it to a measurable function by defining the extension to be constant on the missing subset of null measure. 

By construction $\psi$ is equivariant in the sense of Definition \ref{def:equivariant:map}. Additionally the measurability of $\psi$ follows by the measurability of both $\sigma$ and $q$, and the statement is proved.
\end{proof}

\begin{oss}
The crucial aspect in the previous proof is the existence of a measurable fundamental domain of the $\Gamma$-action on the universal cover $Y$. On the contrary, in the case of boundaries the $\Gamma$-action is not \emph{smooth} in the sense of Zimmer \cite[Definition 2.1.9]{zimmer:libro}), thus it cannot admit a measurable fundamental domain. This is one of the reasons for which proving the existence of boundary maps reveals much more difficult. 
\end{oss}

Now we are going to use the measurable equivariant map $\psi:Y \times \Omega \rightarrow \calX_G$ to define the equivariant family of measures we need. Given almost every point $x \in \Omega$, we can define the \emph{slice associated to the point $x$} as the map
$$
\psi_x:Y \rightarrow \calX_G, \hspace{5pt} \psi_x(a):=\psi(a,x) \ .
$$
Since $\Omega$ is a standard Borel space, by \cite[Lemma 2.6]{fisher:morris:whyte} it follows that the map $\psi_x$ is measurable for almost every $x \in \Omega$. Additionally the equivariance of the map $\psi$ implies the following relation on the slices
$$
\psi_{\gamma.x}(\gamma \ \cdot \ )=\sigma(\gamma,x)\psi_x(\ \cdot \ ) \ ,
$$
for every $\gamma \in \Gamma$ and almost every $x \in \Omega$. If now $\{ \nu_b \}_{b \in \calX_G}$ is the Patterson-Sullivan family defined in Section \ref{sec:patterson:sullivan}, by fixing a number $s>h(Y)$, we can define
\begin{equation}\label{eq:measure:convolution}
\mu^s_{a,x}:=((\psi_x)_\ast(\mu^s_a)) \ast \{\nu_b\}_{b \in \calX_G} \ ,
\end{equation}
where $a \in Y$, $x \in \Omega$ and $\mu^s_a$ is the measure defined in Equation (\ref{eq:equivariant:family}). In a similar way for the convolution appearing in Section \ref{sec:bar:natural}, the convolution of Equation (\ref{eq:measure:convolution}) is defined as follows
$$
\mu^s_{a,x}(U):=\int_{\calX_G} \nu_b(U)d((\psi_x)_\ast(\mu^s_a))(b)=\int_{Y}\nu_{\psi_x(z)}(U)d\mu^s_a(z) \ , 
$$
for every measurable subset $U \subseteq B(G)$. 

We are going now to prove that the family $\{ \mu^s_{a,x} \}_{a \in Y,x \in \Omega}$ is equivariant with respect to the $\sigma$-action.

\begin{lem}\label{lem:equivariant:family}
Let $\Gamma=\pi_1(N)$ be the fundamental group of a closed $n$-dimensional Riemannian manifold $N$ and let $Y$ be its universal cover. Fix $(\Omega,\mu_\Omega)$ a standard Borel probability $\Gamma$-space. 
Suppose $\sigma:\Gamma \times \Omega \rightarrow G$ is a measurable cocycle with measurable equivariant map $\psi:Y \times \Omega \rightarrow \calX_G$ which is $\sigma$-equivariant. Then the family of measures $\{ \mu^s_{a,x} \}_{a \in Y,x \in \Omega}$ defined by Equation (\ref{eq:measure:convolution}) is supported on the Furstenberg boundary $B(G)$ and it is $\sigma$-equivariant, that is 
$$
\mu^s_{\gamma.a,\gamma.x}=\sigma(\gamma,x)_\ast(\mu^s_{a,x}) \ ,
$$
for every $\gamma \in \Gamma$ and almost every $a \in Y$ and $x \in \Omega$. 
\end{lem}

\begin{proof}
Since $\mu^s_{a,x}$ is defined via the convolution with the Patterson-Sullivan family $\{ \nu_b \}_{b \in \calX_G}$ and each of these measures is supported on the Furstenberg-Poisson boundary $B(G)$, the same holds for $\mu^s_{a,x}$. 

We now prove that the family is equivariant. Consider a measurable subset $U \subseteq B(G)$. Then for every $\gamma \in \Gamma$ and almost every $a \in Y$ and $x \in \Omega$ we have
\begin{align*}
\mu^s_{\gamma.a,\gamma.x}(U)&=\int_{\calX_G} \nu_{\psi_{\gamma.x}(b)}(U)d\mu^s_{\gamma.a}(b)=\\ 
&=\int_{\calX_G} \nu_{\psi_{\gamma.x}(b)}(U)d((\gamma_\ast)(\mu^s_a))(b)=\\
&=\int_{\calX_G} \nu_{\psi_{\gamma.x}(\gamma b)}(U)d\mu^s_x(b)=\\
&=\int_{\calX_G} \nu_{\sigma(\gamma,x)\psi_x(b)}(U)d\mu^s_x(b)=\\
&=\int_{\calX_G} (\sigma(\gamma,x)_\ast)(\nu_{\psi_x(b)})(U)d\mu^s_x(b)=(\sigma(\gamma,x)_\ast)(\mu^s_{a,x})(U) \ , 
\end{align*}
where to move from the first line to the second one we use the equivariance of the family $\{\mu^s_a\}_{a \in Y}$, to pass from the second line to the third one we use the direct image theorem, to move from the third line to the fourth one we apply the $\sigma$-equivariance of $\psi$ and finally we use again the equivariance of the Patterson-Sullivan family. The statement now follows. 
\end{proof}

Thanks to the previous lemma we can now prove the existence of natural maps. 

\begin{proof}[Proof of Theorem \ref{teor:natural:map}]
Since by assumption we have a measurable map $\psi:Y \times \Omega \rightarrow \calX_G$ which is $\sigma$-equivariant, we can define the family of measures $\{ \mu^s_{a,x} \}_{a \in Y,x \in \Omega}$ given by Equation (\ref{eq:measure:convolution}). 

For every $s>h(Y)$, we can define the following map
$$
F^s:Y \times \Omega \rightarrow \calX_G \ , 
$$
$$
F^s(a,x):=\barb(\mu^s_{a,x}) = \barb \left( \left( \int_{Y} e^{-sd_Y(a,z)-h(\calX_G)\beta_o(\psi(z,x),\xi)}d\mu(z) \right) d\nu_o(\xi) \right) \ .
$$
Clearly $F^s$ is a well-defined map since the support of the measure $\mu^s_{a,x}$ is the Furstenberg-Poisson boundary because we define it using the convolution with the Patterson-Sullivan family $\{ \nu_b \}_{b \in \calX_G}$. As a consequence of Lemma \ref{lem:equivariant:family} we know that the family $\{ \mu^s_{a,x} \}_{a \in Y, x \in \Omega}$ is $\sigma$-equivariant. The equivariance property implies that
\begin{align*}
F^s(\gamma.a,\gamma.x)&=\barb(\mu^s_{\gamma.a,\gamma.x})=\\
&=\barb(\sigma(\gamma,x)_\ast(\mu^s_{a,x}))=\\
&=\sigma(\gamma,x)\barb(\mu^s_{a,x})=\sigma(\gamma,x)F^s(a,x) \ ,
\end{align*}
for every $\gamma \in \Gamma$ and almost every $a \in Y$ and $x \in \Omega$. This implies the equivariance of $F^s$. 
Now for almost every $x \in \Omega$ we define the \emph{slice associated to the point $x$} as $F^s_x:Y \rightarrow\calX_G, \hspace{5pt} F^s_x(a):=F^s(a,x)$. Since $\Omega$ is a standard Borel space, by \cite[Lemma 2.6]{fisher:morris:whyte} it follows that the function $\widehat{F}^s:\Omega \rightarrow \textup{Meas}(Y,\calX_G)$ is measurable, and hence $F^s_x$ is measurable for almost every $x \in \Omega$. 

We are going to prove that for almost every $x \in \Omega$ the map $F^s_x$ has actually more regularity. Recall that the implicit Equation (\ref{eq:implicit:barycenter}) satisfied by the barycenter, in this particular context, it becomes
\begin{equation}\label{eq:natural:implicit}
\int_{B(G)} d\beta_o|_{(F^s_x(a),\xi)} (\ \cdot \ )d\mu^s_{a,x}(\xi)=0 \ .
\end{equation}
Following either Besson-Courtois-Gallot \cite{bcg95,bcg96,bcg98} or Connell-Farb \cite{connellfarb2,connellfarb1} we have that the implicit equation above implies that the map $F^s_x$ is actually differentiable for almost every $x \in \Omega$.  

The last thing we want to prove is the uniform bound on the Jacobian of $F^s_x$. We are going to follow the line of the proof of \cite[Theorem A]{connellfarb2}. To do this we need to differentiate again Equation (\ref{eq:natural:implicit}). In this way, for almost every $x \in \Omega$ and every $a \in Y$ we obtain
\begin{align}\label{eq:implicit:differentiated}
&\int_{B(G)} \nabla d\beta_o|_{(F^s_x(a),\xi)}(D_aF^s_x(u),v)d\mu^s_{a,x}(\xi)=\\
s\int_{Y} &\int_{B(G)} d\beta_o|_{(F^s_x(a),\xi)}(v) \cdot \langle \grad_a d_Y(a,z), u \rangle d\nu_{\psi_x(z)}(\xi)d\mu^s_a(z) \nonumber \ ,
\end{align}
where $u \in T_aY, v \in T_{F_x^s(a)}\calX_G$ and $\psi_x:Y \rightarrow \calX_G$ is the slice of $\psi$ associated to $x \in \Omega$. Here $\nabla$ is the Levi-Civita connection associated to the standard Riemannian metric on $\calX_G$ and $\grad_a$ is the Riemannian gradient computed at the point $a \in Y$. If we now consider the determinant of Equation (\ref{eq:implicit:differentiated}) we get
\begin{equation}\label{eq:implicit:determinant}
\jac_a F^s_x= s^n \frac{\det \left(\int_{Y} \int_{B(G)} d\beta_o|_{(F^s_x(a),\xi)}( \ \cdot \ ) \cdot \langle \grad_a d_Y(a,z),\ \cdot \ \rangle d\nu_{\psi_x(z)}(\xi) d\mu^s_a(z)\right)}{\det \left(\int_{B(G)} \nabla d\beta_o|_{(F^s_x(a),\xi)}(\ \cdot \ ,\ \cdot \ )d\mu^s_{a,x}(\xi)\right)}
\end{equation}
where on both the nominator and the denominator we considered the determinant of the bilinear forms which appear in Equation (\ref{eq:implicit:differentiated}). By applying the Cauchy-Schwarz inequality with respect to the numerator of the right-hand side of Equation (\ref{eq:implicit:determinant}) we get
\begin{equation}\label{eq:implicit:cauchy}
\jac_a F^s_x \leq s^n \frac{\det \left(\int_{B(G)} \left(d\beta_o|_{(F^s_x(a),\xi)}( \ \cdot \ ) \right)^2d\mu^s_{a,x}(\xi)\right)^{\frac{1}{2}} \det \left(\int_{Y} \langle \grad_a d_Y(a,z), \ \cdot \ \rangle^2d\mu^s_{a}(z)\right)^{\frac{1}{2}}}{\det \left(\int_{B(G)} \nabla d\beta_o|_{(F^s_x(a),\xi)}(\ \cdot \ ,\ \cdot \ )d\mu^s_{a,x}(\xi)\right)}
\end{equation}

Since the trace satisfies $\tr \langle \grad_a d_Y(a,z), \ \cdot \ \rangle^2=1$ outside a measure zero set, it holds 
$$
\det \left(\int_{Y} \langle \grad_a d_Y(a,z), \ \cdot \ \rangle^2d\mu^s_{a}(z)\right)^{\frac{1}{2}} \leq \left(\frac{1}{\sqrt{n}}\right)^n \ ,
$$
Inequality (\ref{eq:implicit:cauchy}) boils down to
\begin{equation}\label{eq:implicit:bilinear:forms}
\jac_a F^s_x \leq \left(\frac{s}{\sqrt{n}}\right)^n \frac{\det \left(\int_{B(G)} \left(d\beta_o|_{(F^s_x(a),\xi)}( \ \cdot \ ) \right)^2d\mu^s_{a,x}(\xi)\right)^{\frac{1}{2}}}{\det \left(\int_{B(G)} \nabla d\beta_o|_{(F^s_x(a),\xi)}(\ \cdot \ ,\ \cdot \ )d\mu^s_{a,x}(\xi)\right)} \ .
\end{equation}

Now, following the same computation of \cite[Section 4.2]{connellfarb2} one can prove that without loss of generality it is possible to assume $G$ irreducible and then the desired estimate follows now by \cite{connellfarb1,connellfarb4}, as desired. 
\end{proof}

\begin{oss}
It is worth noticing that for constructing the $\sigma$-equivariant family $\{ \mu^s_{a,x} \}_{a \in Y, x \in \Omega}$ and hence for defining the map $F^s:Y \times \Omega \rightarrow \calX_G$ we exploited the existence of a measurable map $\psi:Y \times \Omega \rightarrow \calX_G$ and not of a boundary map as in \cite{savini4:articolo}.
\end{oss}

So far we have shown that, given a measurable cocycle $\sigma:\Gamma \times \Omega \rightarrow G$ which admits a measurable $\sigma$-equivariant map $\psi: Y \times \Omega \rightarrow \calX_G$, for every $s>h(Y)$, there exists a map 
$$
F^s:Y \times \Omega \rightarrow \calX_G \ ,
$$ 
which is $\sigma$-equivariant and its slices are differentiable. It is quite natural to ask what can happen if $\sigma$ is actually a measurable cocycle induced by a representation $\rho:\Gamma \rightarrow G$. More precisely one could ask which relation exists between the natural map defined in Theorem \ref{teor:natural:map} and the natural map defined by Connell-Farb \cite{connellfarb2}. This is exactly the content of the following

\begin{prop}\label{prop:natural:rep}
Let $\Gamma=\pi_1(N)$ be the fundamental group of a closed Riemannian manifold $N$ whose universal cover is $Y$. Consider $\rho:\Gamma \rightarrow G$ a representation. Let $\psi:Y \rightarrow \calX_G$ be a measurable $\rho$-equivariant map. Denote by $F^s:Y \rightarrow \calX_G$ and by $\sigma_\rho:\Gamma \times \Omega \rightarrow \calX_G$ the natural map and the measurable cocycle associated to $\rho$, respectively. Then for every $s>h(Y)$ the natural map associated to $\sigma_\rho$ is given by
$$
\widetilde{F}^s:Y \times \Omega \rightarrow \calX_G, \hspace{5pt} \widetilde{F}^s(a,x):=F^s(a) \ .
$$
\end{prop}

\begin{proof}
Starting from the map $\psi:Y \rightarrow \calX_G$ we can define the following measurable map
$$
\widetilde{\psi}:Y \times \Omega \rightarrow \calX_G, \hspace{5pt} \widetilde{\psi}(a,x):=\psi(a) \ ,
$$
which is clearly $\sigma_\rho$-equivariant, since $\psi$ is $\rho$-equivariant. In particular for every $x \in \Omega$ we have the equality $\psi_x=\psi$. By applying the definition which appears in the proof of Theorem \ref{teor:natural:map} we have that
\begin{align*}
\widetilde{F}^s(a,x)&=\barb( ((\psi_x)_\ast(\mu^s_a)) \ast \{ \nu_b \}_{b \in \calX_G})=\\
&=\barb( (\psi)_\ast(\mu^s_a)) \ast \{ \nu_b \}_{b \in \calX_G})=F^s(a) \ ,
\end{align*}
and the statement follows. 
\end{proof}

The proposition above can be compared with \cite[Proposition 3.2]{savini4:articolo}, which should be interpreted as analogous statement for rank-one Lie groups. We conclude the section by showing how the natural map $F^s$ change along the $G$-cohomology class of a fixed measurable cocycle (compare with \cite[Proposition 3.3]{savini4:articolo}). 

\begin{prop}\label{prop:natural:cohomology}
Let $\Gamma=\pi_1(N)$ be the fundamental group of a closed Riemannian manifold $N$ whose universal cover is $Y$. Let $\sigma:\Gamma \times \Omega \rightarrow G$ be a measurable cocycle with measurable $\sigma$-equivariant map $\psi:Y \times \Omega \rightarrow \calX_G$. Then, given a measurable map $f: \Omega \rightarrow G$, the natural map associated to the cocycle $f.\sigma$ is given by
$$
f.F^s:Y \times \Omega \rightarrow \calX_G, \hspace{5pt} (f.F^s)(a,x)=f(x)^{-1}F^s(a,x) \ .
$$ 
\end{prop}

\begin{proof}
If $\psi:Y \times \Omega \rightarrow \calX_G$ is a measurable $\sigma$-equivariant map, then the map 
$$
f.\psi:Y \times \Omega \rightarrow \calX_G, \hspace{5pt} (f.\psi)(a,x):=f(x)^{-1}\psi(a,x) \ ,
$$
defined for almost every $a \in Y$ and $x \in \Omega$, is clearly measurable and $f.\sigma$-equivariant. 

Using the definition of natural map we have 
\begin{align*}
(f.F^s)(a,x)&=\barb( ((f.\psi_x)_\ast (\mu^s_a)) \ast \{\nu_b \}_{b \in \calX_G})=\\ 
&=f(x)^{-1}\barb(((\psi_x)_\ast (\mu^s_a)) \ast \{\nu_b \}_{b \in \calX_G}) =f(x)^{-1}F^s(a,x) \ ,
\end{align*}
where we used the $G$-equivariance of the barycenter (see Section \ref{sec:bar:natural}) to pass from the first line to the second one. Hence the claim follows. 
\end{proof}

\section{Volume of equivariant maps}\label{sec:volume}

Let $G$ be a semisimple Lie group of non-compact factor and denote by $\calX_G$ the Riemannian symmetric space associated to it. Suppose $\textup{dim}(\calX_G)=n$. Consider a closed $n$-dimensional complete Riemannian manifold $N$ with fundamental group $\Gamma=\pi_1(N)$ and universal cover $Y$. Fix $(\Omega,\mu_\Omega)$ a standard Borel probability $\Gamma$-space. Given a measurable cocycle $\sigma:\Gamma \times \Omega \rightarrow G$, in this section we are going to deal with a measurable $\sigma$-equivariant map $\Phi:Y \times \Omega \rightarrow \calX_G$. Under suitable hypothesis on such a map, we are going to define the notion of volume associated it. To properly define this notion we will need to assume a uniform bound on the Jacobian of the slices. This will allow to consider the pullback of the volume form on $\calX_G$ and to integrate it  first along the probability space $\Omega$ and then on the manifold $N$. The volume of equivariant map will enable us to state a degree theorem for equivariant maps similar to the one given by \cite[Proposition 1.3]{savini4:articolo}. 

Let $\Phi:Y \times \Omega \rightarrow \calX_G$ be a measurable $\sigma$-equivariant map. For almost every $x \in \Omega$ we define the \emph{slice associated to $x$}, $\Phi_x: Y \rightarrow \calX_G$ and we are going to assume that these maps are smooth for almost every $x \in \Omega$. Hence it makes sense to speak about the Jacobian $\jac_aF_x$ for every $a \in \calX_G$. We are going to say that $\Phi$ is \emph{essentially bounded}, or it has \emph{essentially bounded slices}, if there exists $C>0$ such that 
$$
\jac_a \Phi_x < C \ ,
$$
for every $a \in \calX_G$ and almost every $x \in \Omega$. Assume now that $\Phi$ is essentially bounded. If we denote by $\omega_G$ and $\omega_Y$ the Riemannian volume forms on $\calX_G$ and $Y$, respectively, then, imitating what the author did in \cite{savini4:articolo}, we can consider 
$$
\omega_x:=\Phi^\ast_x \omega_G = \textup{Jac}_a\Phi_x \cdot \omega_Y \ ,
$$
for almost every $x \in \Omega$. Since the Jacobian is uniformly bounded and $\Omega$ is a probability space, we can consider the integral
$$
\widehat{\omega}_\Phi:=\int_\Omega \omega_x d\mu_\Omega(x) . 
$$ 
More precisely, given a $n$-tuple $\{ u_1,\ldots,u_n\}$ of vectors in $T_aY$, we have
\begin{align*}
\widehat{\omega}_\Phi(u_1,\ldots,u_n):&=\int_\Omega \omega_x(u_1,\ldots,u_n)d\mu_\Omega(x)=\\
&=\int_{\Omega} \omega_G(D_a\Phi_x(u_1),\ldots,D_a\Phi_x(u_n))d\mu_\Omega(x) \ . 
\end{align*}

The same strategy exposed in \cite[Section 4]{savini4:articolo} shows that the form $\widehat{\omega}_\Phi$ is a smooth $\Gamma$-invariant differential form on $Y$ and hence it induces a differential form $\omega_\Phi \in \Omega^n(N)$. This allows us to give the following

\begin{deft}
Let $\Gamma$ be the fundamental group of a closed $n$-dimensional Riemannian manifold $N$ whose universal cover is $Y$. Fix a standard Borel probability $\Gamma$-space $(\Omega,\mu_\Omega)$. If we have $\sigma:\Gamma \times \Omega \rightarrow G$ a measurable cocycle, we denote the set
\begin{align*}
\mathscr{D}(\sigma):=\{ \Phi: Y \times \Omega \rightarrow \calX_G| &\text{$\Phi$ essentially bounded $\sigma$-equivariant map}\\
& \text{with differentiable slices} \}
\end{align*}
Given an element $\Phi \in \mathscr{D}(\sigma)$ we define the \emph{volume of the map $\Phi$} as 
$$
\vol(\Phi):=\int_{N} \omega_\Phi=\int_{N} \int_\Omega \omega_x d\mu_\Omega(x) \ .
$$
\end{deft}

\begin{oss}
In a similar way for what happens in \cite{savini4:articolo} for rank-one Lie groups, if a measurable cocycle $\sigma: \Gamma \times \Omega \rightarrow G$ admits a measurable map $\psi:Y \times \Omega \rightarrow \calX_G$ which is $\sigma$-equivariant, then the set $\mathscr{D}(\sigma)$ is not empty. Indeed, for $s>h(Y)$, the map $F^s:Y \times \Omega \rightarrow \calX_G$ has differentiable slices by Theorem \ref{teor:natural:map}. Additionally, by the same statement, we know that there exists a uniform $C>0$ such that
$$
\jac_aF_x^s \leq C \left(\frac{s}{h(\calX_G)}\right)^n \ ,
$$
for every $a \in Y$ and almost every $x \in \Omega$.  This means exactly that the map $F^s$ is essentially bounded and hence $F^s \in \mathscr{D}(\sigma)$. Moreover, the bound on the Jacobian implies
\begin{equation}\label{volume:inequality}
\vol(F^s) \leq C \left( \frac{s}{h(\calX_G)} \right)^n \vol(N) \ .
\end{equation}
\end{oss}

\begin{oss}
In the particular case when $\Gamma$ is a lattice in $G$ and the cocycle is cohomologous to a representation $\rho:\Gamma \rightarrow G$, the Estimate \ref{volume:inequality} can be improved. Indeed, in that case, by Proposition \ref{prop:natural:rep} the volume of $F^s$ boils down to the volume of the representation discussed in \cite{kim:kim:14} and for that volume it holds
$$
\vol(F^s)=\vol(\rho) \leq \vol(\Gamma \backslash \calX_G). 
$$
By \cite[Theorem 1.2]{kim:kim:14} we know that the equality is attained if and only if the representation is discrete and faithful. As a consequence we get that 
$$
\vol(F^s)=\vol(\Gamma \backslash \calX_G) \ ,
$$
for the natural map $F^s$ associated to a Zariski dense cocycle $\sigma:\Gamma \times \Omega \rightarrow \calX_G$ (recall that such a cocycle is cohomologous to a discrete and faithful representation by Zimmer Superrigidity Theorem \cite{zimmer:annals}). 
\end{oss}

We want to conclude the section by showing a suitable version of mapping degree theorem for measurable equivariant maps associated to cocycles. In order to do this we are going to follow both \cite[Section 6]{moraschini:savini} and \cite[Section 5]{savini4:articolo} to introduce the notion of \emph{pullback of a measurable cocycle along a continuous map}. Let $N,M$ be a closed $n$-dimensional Riemannian manifolds with fundamental groups $\Gamma=\pi_1(N), \Lambda=\pi_1(M)$ and universal covers $Y,X$, respectively. Let $f:N \rightarrow M$ be a smooth function with non-vanishing degree. Suppose that the Jacobian of $f$ is uniformly bounded. For instance this is the case when $\Lambda$ is a torsion-free uniform lattice in the isometry group of a non-positively curved symmetric space by the existence of natural maps. We denote by $\pi_1(f):\Gamma \rightarrow \Lambda$ the induced map on the fundamental groups. Given a measurable cocycle $\sigma:\Lambda \times \Omega \rightarrow G$, where $\Omega$ is the usual standard Borel probability $\Lambda$-space, we can define the pullback cocycle as follows
$$
f^\ast \sigma: \Gamma \times \Omega \rightarrow G, \hspace{5pt} f^\ast \sigma(\gamma,x):=(\pi_1(f)(\gamma),x) \ ,
$$
where the structure of $\Gamma$-space on $\Omega$ is induced by $\pi_1(f)$. As proved in \cite[Lemma 6.1]{moraschini:savini} the previous cocycle is well-defined. 

Let $\widetilde{f}:Y \rightarrow X$ the lift of the map $f$ to the universal covers. The existence of such a map allows to consider the pullback of measurable $\sigma$-equivariant map with respect to the continuous map $f$. More precisely, given an element $\Phi \in \mathscr{D}(\sigma)$ we can define the following map
$$
f^\ast \Phi: Y \times \Omega \rightarrow \calX_G, \hspace{5pt} f^\ast \Phi(a,x):=\Phi(\widetilde{f}(a),x) \ , 
$$
for every $a \in Y$ and almost every $x \in \Omega$. Notice that $f^\ast \Phi \in \mathscr{D}(f^\ast \sigma)$ by the boundedness assumption on the Jacobian of $f$, and hence it has a well-defined volume. 

Having introduced all the notion we need, we are ready to prove the main proposition. 

\begin{proof}[Proof of Proposition \ref{prop:degree:map}]
The proof follows the line of \cite[Proposition 1.3]{savini4:articolo}. Let $\omega_N$ and $\omega_M$ the volume form associated to the Riemannian structure on $N$ and $M$, respectively. By changing suitably the orientation of either $N$ or $M$, we can suppose without loss of generality that the degree $\deg(f)$ is positive. 

By definition of volume of equivariant maps, we have
\begin{align*}
\vol(f^\ast \Phi)&=\int_{N} \int_\Omega (f^\ast \Phi)_x^\ast \omega_G \ d\mu_\Omega(x)=\int_N \left( \int_\Omega \jac_a (f^\ast \Phi_x)d\mu_\Omega(x) \right) \omega_N = \\
&=\int_N \jac_a f \left( \int_\Omega \jac_{\widetilde{f}(a)}\Phi_x d\mu_\Omega(x) \right) \omega_N \ ,
\end{align*}
where we used the equivariance of the map $\widetilde{f}:Y \rightarrow X$ to move from the first line to the second one. If we now apply the coarea formula we obtain
\begin{align*}
\int_N \jac_a f \left( \int_\Omega \jac_{\widetilde{f}(a)}\Phi_x d\mu_\Omega(x) \right) \omega_M&=\int_M \calN(b) \left( \int_\Omega \jac_b\Phi_x d\mu_\Omega(x) \right) \omega_M \geq \\
&\geq \deg(f) \cdot \int_M \left(\int_\Omega \jac_b \Phi_x d\mu_\Omega(x) \right) \omega_M = \\
&= \deg(f) \cdot \vol(\Phi) \ ,
\end{align*}
where we denoted by $\calN(b)$ the cardinality of the set
$$
\calN(b)=\textup{card}((\widetilde{f})^{-1}(b)) \ .
$$
The statement now follows. 
\end{proof}

%**************************
% Bibliografia
%**************************

\bibliographystyle{amsalpha}
\bibliography{biblionote}

\end{document}